\documentclass{article}
\usepackage[utf8]{inputenc}
\usepackage{amsmath,amssymb,amsthm, amsfonts}
\usepackage{graphicx}
\usepackage{tkz-euclide}
\usepackage{mathrsfs}
\usepackage{xcolor}

\title{Introduction to Robust Power Domination}
\author{Beth Bjorkman\thanks{Air Force Research Laboratory, beth.morrison@us.af.mil} \and 
        Esther Conrad\thanks {Iowa State University, edconrad@iastate.edu}}
\date{December 15, 2022}
\newtheorem{thm}{Theorem}[section]
\newtheorem{cor}[thm]{Corollary}

\newtheorem{prop}[thm]{Proposition}

\newtheorem{obs}[thm]{Observation}

\theoremstyle{definition}
\newtheorem{rem}[thm]{Remark}
\theoremstyle{definition}
\newtheorem{defn}[thm]{Definition}
\theoremstyle{definition}

\newcommand{\lb}{\left\{}
\newcommand{\rb}{\right\}}
\newcommand{\lf}{\left\lfloor}
\newcommand{\rf}{\right\rfloor}
\newcommand{\lc}{\left\lceil}
\newcommand{\rc}{\right\rceil}
\newcommand{\Mod}[1]{\ (\mathrm{mod}\ #1)}
\newcommand{\ben}{\begin{enumerate}}
\newcommand{\een}{\end{enumerate}}
\newcommand{\lp}{\left (}
\newcommand{\rp}{\right )}

\newcommand{\gpk}[1]{\ddot{\gamma}_P^k\left(#1\right)}
\newcommand{\gpkset}{\ddot{\gamma}_P^k\text{-set}}
\newcommand{\gpkplus}[1]{\ddot{\gamma}_P^{k+1}\left(#1\right)}

\newcommand{\gpkother}[2]{\ddot{\gamma}_P^{#1}\left(#2\right)}
\newcommand{\gpkotherset}[1]{\ddot{\gamma}_P^{#1}\text{-set}}

 \newcommand{\gp}[1]{\gamma_P\left(#1\right)}
\renewcommand{\sp}[1]{\operatorname{sp}\left(#1\right)}
\newcommand{\pmus}[1]{\operatorname{\#PMU}\left(#1\right)}
\newcommand{\pmuss}[2]{\operatorname{\#PMU}_{#1}\left(#2\right)}
\renewcommand{\deg}[2]{\operatorname{deg}_{#2}\left(#1\right)}

\newcommand{\bigpds}[1]{s \lp #1\rp}

\newcommand{\ds}{\displaystyle}

\newcommand{\vftk}[1]{{\gamma}_P^k\left(#1\right)}
\newcommand{\vftother}[2]{{\gamma}_P^{#1}\left(#2\right)}

\colorlet{outline}{blue!100}
\colorlet{pmu}{blue!30!}

\begin{document}
\maketitle
\begin{abstract}
    
   Sensors called phasor measurement units (PMUs) are used to monitor the electric power network. The power domination problem seeks to minimize the number of PMUs needed to monitor the network. We extend the power domination problem and consider the minimum number of sensors and appropriate placement to ensure monitoring when $k$ sensors are allowed to fail with multiple sensors allowed to be placed in one location. That is, what is the minimum multiset of the vertices, $S$, such that for every $F\subseteq S$ with $|F|=k$, $S\setminus F$ is a power dominating set. Such a set of PMUs is called a \emph{$k$-robust power domination set}. This paper generalizes the work done by Pai, Chang and Wang in 2010 on vertex-fault-tolerant power domination, which did not allow for multiple sensors to be placed at the same vertex. We provide general bounds and determine the $k$-robust power domination number of some graph families.
    \end{abstract}
\section{Introduction}

The power domination problem seeks to find the placement of the mimimum number of sensors called phasor measurement units (PMUs) needed to monitor an electric power network. In \cite{hhhh02}, Haynes et al. defined the power domination problem in graph theoretic terms by placing PMUs at a set of initial vertices and then applying observation rules to the vertices and edges of the graph. This process was simplified by Brueni and Heath in \cite{bh05}.


Pai, Chang, and Wang \cite{pcw10} generalized power domination to create \emph{vertex-fault-tolerant power domination} in 2010 to model the possibility of sensor failure. The $k$-\textit{fault-tolerant power domination problem} seeks to find the minimum number of PMUs needed to monitor a power network (and their placements) given that any $k$ of the PMUs will fail. The vertex containing the failed PMU remains in the graph, as do its edges; it is only the PMU that fails. This generalization allows for the placement of only one PMU per vertex.

We consider the related problem of the minimum number of PMUs needed to monitor a power network given that $k$ PMUs will fail \emph{but also allow for multiple PMUs to be placed at a given vertex}. We call this \emph{PMU-defect-robust power domination}, as it is not the vertices that cause a problem with monitoring the network, but the individual PMUs themselves. This models potential synchronization issues, sensor errors, or malicious interference with the sensor outputs.

    To demonstrate the difference between vertex-fault-tolerant power domination and PMU-defect-robust power domination and how drastic the difference between these two parameters can be, consider the star on $16$ vertices with $k= 1$, shown in Figure \ref{fig:starmulti}. Notice that in vertex-fault-tolerant power domination, if one PMU is placed in the center of the star and this PMU fails, then all but one of the leaves must have PMUs in order to still form a power dominating set. However, with PMU-defect-robust power domination, placing two PMUs in the center is sufficient to ensure that even if one PMU fails, the power domination process will still observe all of the vertices. \\
    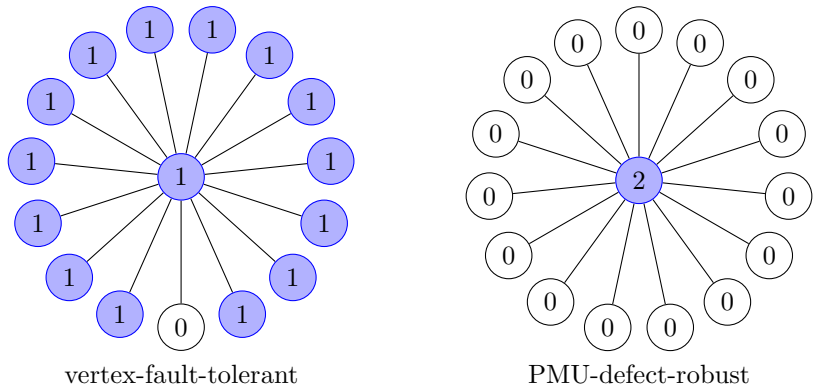
\begin{figure}[htbp]
    \begin{center}
    \begin{tabular}{ccc}
    \begin{tikzpicture}[shorten >=9pt, shorten <=9pt, scale=1]
    \foreach \a in {1,2,...,15}{
        \coordinate (\a) at (\a*360/15-90:2cm);
    }
    \foreach \b in {1,2,...,14}{
        \draw[shorten >=9pt, shorten <=9pt] (0,0) -- (\b);
        \node[circle, draw=outline, fill=pmu] at (\b) {$1$};}
    \foreach \b in {15}{
        \draw[shorten >=9pt, shorten <=9pt] (0,0) -- (\b);
        \node[circle, draw=black, fill=white] at (\b) {$0$};}    
    \node[circle, draw=outline, fill=pmu] at (0,0) {$1$};
    \end{tikzpicture}
    &
    \hspace{0.25in}
    &
    \begin{tikzpicture}[shorten >=9pt, shorten <=9pt, scale=1]
    \foreach \a in {1,2,...,15}{
        \coordinate (\a) at (\a*360/15+90:2cm);
    }
    \foreach \b in {1,2,...,15}{
        \draw[shorten >=9pt, shorten <=9pt] (0,0) -- (\b);
        \node[circle, draw=black, fill=white] at (\b) {$0$};}
    \node[circle, draw=outline, fill=pmu] at (0,0) {$2$}; 
    \end{tikzpicture}\\
    vertex-fault-tolerant &\hspace{0.25in}& PMU-defect-robust \\
    \end{tabular}
    \end{center}
    \caption{A minimum vertex-fault-tolerant power dominating set and a minimum PMU-defect-robust power dominating set shown for a star when $k=1$.}
    \label{fig:starmulti}
    \end{figure}

    In Section \ref{sec:prelimkrpds}, we review definitions from past work and formally define PMU-defect-robust power domination. We also include some basic results in that section. Section \ref{sec:boundskrpds} consists of general bounds for $k$-robust power domination and in Section \ref{sec:completebipartite} we demonstrate the tightness of these bounds with a family of complete bipartite graphs. In Section \ref{sec:trees} we establish the $k$-robust power domination number for trees. Section \ref{sec:concrem} contains concluding remarks, including suggestions for future work.

\section{Preliminaries}\label{sec:prelimkrpds}
    
We begin by giving relevant graph theory definitions. Then we define power domination, vertex-fault-tolerant power domination, and PMU-defect-robust power domination. Finally, we include useful properties of the floor and ceiling functions.

\subsection{Graph Theory}

A graph $G$ is a set of vertices, $V(G)$, and a set of edges, $E(G)$.  Each (unordered) edge consists of a set of two distinct vertices; the edge $\{u,v\}$ is often written as $uv$. When $G$ is clear, we write $V=V(G)$ and $E=E(G)$. A \emph{path} from $v_1$ to $v_{\ell+1}$ is a sequence of vertices and edges $v_1, e_1, v_2, e_2, \ldots, v_\ell, e_\ell, v_{\ell+1}$ so that the $v_i$ are distinct vertices and $v_i\in e_i$ for all $i$ and $v_i\in e_{i-1}$ for all $i\geq 2$. Such a path has \emph{length} $\ell$. The \emph{distance} between vertices $u$ and $v$ is the minimum length of a path between $u$ and $v$. A graph $G$ is \emph{connected} if there is a path from any vertex to any other vertex. \emph{Throughout what follows, we consider only graphs that are connected.}

We say that vertices $u$ and $v$ are \emph{neighbors} if $uv\in E$. The \emph{neighborhood} of $u\in V$ is the set containing all neighbors of $u$ and is denoted by $N(u)$. The \emph{closed neighborhood} of $u$ is $N[u]=N(u)\cup \{u\}$. The \emph{degree} of a vertex $u\in V$ is the number of edges that contain $u$, that is, $\deg{u}{G} = |N(u)|$. When $G$ is clear, we omit the subscript.  The \emph{maximum degree} of a graph $G$ is $\Delta\left(G\right) = \ds \max_{v\in V} \deg{v}{}$.

A \emph{subgraph} $H$ of a graph $G$ is a graph such that $V(H)\subseteq V(G)$ and $E(H)\subseteq E(G)$. An \emph{induced subgraph} $H$ of a graph $G$, denoted $H=G[V(H)]$, is a graph with vertex set $V(H)\subseteq V(G)$ and edge set $E(H)=\{ uv : u,v\in V(H) \text{ and } uv\in E(G)\}$.
 
We refer the reader to \textit{Graph Theory} by Diestel \cite{diestelbook} for additional graph terminology not detailed here.

\subsection{Power domination, vertex-fault-tolerant power domination, and PMU-defect-robust power domination}

What follows is an equivalent statement of the power domination process as defined in \cite{hhhh02}, and established by \cite{bh05}.

The \emph{power domination process} on a graph $G$ with initial set $S\subseteq V$ proceeds recursively by:
\begin{enumerate}
\item $B = \ds \bigcup_{v\in S} N[v]$
\item While there exists $v\in B$ such that exactly one neighbor, say $u$, of $v$ is \emph{not} in $B$, add $u$ to $B$. 
\end{enumerate}
Step 1 is referred to as the \emph{domination step} and each repetition of step 2 is called a \emph{zero forcing step}.  During the process, we say that a vertex in $B$ is \emph{observed} and a vertex not in $B$ is \emph{unobserved}. A \emph{power dominating set} of a graph $G$ is an initial set $S$ such that $B=V(G)$ at the termination of the power domination process. The \emph{power domination number} of a graph $G$ is the minimum cardinality of a power dominating set of $G$ and is denoted by $\gp{G}$. 

In \cite{pcw10}, Pai, Chang, and Wang define the following variant of power domination. For a graph $G$ and an integer $k$ with $0\leq k \leq |V|$, a set $S\subseteq V$ is called a \emph{$k$-fault-tolerant power dominating set of $G$} if $S\setminus F$ is still a power dominating set of $G$ for any subset $F\subseteq V$ with $|F|\leq k$.  The \emph{$k$-fault-tolerant power domination number}, denoted by $\vftk{G}$, is the minimum cardinality of a $k$-fault-tolerant power dominating set of $G$. 

While $k$-fault-tolerant power domination allows us to examine what occurs when a previously chosen PMU location is no longer usable (yet the vertex remains in the graph), it is also interesting to study when an individual PMU fails. That is, allow for multiple PMUs to be placed at the same location and consider if a subset of the PMUs fail. This also avoids issues with poorly connected graphs, such as in Figure \ref{fig:starmulti}, where $\vftother{1}{G}$ may be close to the number of vertices of $G$. Thus we define \emph{PMU-defect-robust power domination} as follows.  

\begin{defn}
For a given graph $G$ and integer $k\geq 0$, we say that a multiset $S$, each of whose elements is in $V$, is a \emph{$k$-robust power dominating set} of $G$ if $S\setminus F$ is a power dominating set of $G$ for any submultiset $F$ of $S$ with $|F|=k$. We shorten $k$-robust power dominating set of $G$ to $k$-rPDS of $G$. The size of a minimum $k$-rPDS is denoted by $\gpk{G}$ and such a multiset is also referred to as a $\gpkset$ of $G$. The \emph{number of  PMUs} at a vertex $v\in S$ is its multiplicity in $S$, denoted by $\pmuss{S}{v}$, or when $S$ is clear, by $\pmus{v}$. \end{defn}

There are several observations that one can quickly make.

\begin{obs}\label{obs:comparing}
    Let $G$ be a graph and $k\geq 0$. Then
    \begin{enumerate}
        \item $\gpkother{0}{G} = \vftother{0}{G}=\gp{G}$,
        \item $\gpk{G} \leq \vftk{G}$,
        \item $\gp{G}=1$ if and only if $\gpk{G} = k+1$.
    \end{enumerate}     
\end{obs}




For any minimum $k$-rPDS, having more than $k+1$ PMUs at a single vertex is redundant. 

\begin{obs}\label{obs:numpmusleqk+1}
Let $G$ be a graph and $k\geq 0$. If $S$ is a $\gpkset$ of $G$, then for all $v\in S$ we have $\pmus{v} \leq k+1$.
\end{obs}

\subsection{Floor and ceiling functions}
Throughout what follows, recall the following rules for the floor and ceiling functions. Most can be found in Chapter 3 in \cite{knuthbook} and we provide proofs for the rest.

\begin{prop}{\rm \cite[Equation 3.11]{knuthbook}}\label{prop:ceilfracfix}
If $m$ is an integer, $n$ is a positive integer, and $x$ is any real number, then
\[ \left\lceil  \frac{\left\lceil x \right \rceil+m}{n} \right\rceil =\left\lceil  \frac{x+m}{n} \right\rceil. \]
\end{prop}

\begin{prop}{\rm \cite[Ch. 3 Problem 12]{knuthbook}}\label{prop:ceiltofloor}
If $m$ is an integer and $n$ is a positive integer, then
\[ \left\lceil  \frac{m}{n} \right\rceil =\left\lfloor  \frac{m-1}{n} \right\rfloor +1. \]
\end{prop}

\begin{prop}{\rm \cite[Equation 3.4]{knuthbook}}\label{prop:ceilneg}
For any real number $x$, $\lceil -x \rceil = - \lfloor x \rfloor$.
\end{prop}

\begin{prop} \label{prop:ceilfuncbound}
If $x$ and $y$ are real numbers then
\[\lceil x \rceil + \lceil y \rceil -1 \leq \lceil x+y \rceil.\]
\end{prop}

\begin{proof}
Observe that
$
\left \lceil x \right\rceil -1 + \left \lceil y \right \rceil -1 < x+y
$
and so
$\left \lceil x \right\rceil + \left \lceil y \right \rceil -2 < \left \lceil x+y\right\rceil $
which is a strict inequality of integers, so
$\lceil x \rceil + \lceil y \rceil -1 \leq \lceil x+y \rceil$.
\end{proof}

We can repeatedly apply the inequality in Proposition \ref{prop:ceilfuncbound} to obtain
\begin{cor} \label{cor:ceilfuncmultbound}
If $x$ is a real number and $a$ is a positive integer then
\[a \lceil x \rceil \leq \lceil ax \rceil +a-1.\] 
\end{cor}

    \section{General bounds}\label{sec:boundskrpds}
    
    A useful property of robust power domination is the subadditivity of the parameter with respect to $k$. This idea is established in the next three statements. 
    
    \begin{prop}\label{prop:incr}
    Let $k\geq 0$. For any graph $G$, $\gpk{G} +1 \leq \gpkplus{G}$.
    \end{prop}
    
    \begin{proof}
    Consider a $\gpkotherset{k+1}$, $S$, of $G$. Let $v\in S$. Create $S'=S\setminus\{v\}$, that is, $S'$ is $S$ with one fewer PMU at $v$. Observe that for any $F'\subseteq S'$ with $|F'|=k$, we have $F'\cup \{v\} \subseteq S$ and $|F'\cup \{v\}|=k+1$. Hence $S\setminus \lp F'\cup \{v\} \rp$ is a power dominating set of $G$. Thus, for any such $F'$, we have $\lp S\setminus \{v\} \rp\setminus F' = S'\setminus F'$ is a power dominating set of $G$. Therefore, $S'$ is a $k$-robust power dominating set of $G$ of size $|S|-1$. 
    \end{proof}
    
    Proposition \ref{prop:incr} can be applied repeatedly to obtain the next result.
    \begin{cor}\label{cor:incrj}
    Let $k\geq 0$ and $j\geq 1$. For any graph $G$, \[\gpk{G} + j \leq \gpkother{k+j}{G}.\]
    \end{cor}
    Corollary \ref{cor:incrj} implies the lower bound in the next proposition. The upper bound follows from taking $k+1$ copies of any minimum power dominating set for $G$ to form a $k$-rPDS.
    \begin{prop}\label{prop:basicbounds} Let $k\geq 0$. For any graph $G$,
    \[\gp{G}+k \leq \gpk{G} \leq (k+1)\gp{G}.\]
    \end{prop}
    
Observe that if $\gp{G} =1$ for any graph $G$, both Observation \ref{obs:comparing} and Proposition \ref{prop:basicbounds} demonstrate that $\gpk{G} = k+1.$ 
    
Haynes et al. observed in \cite[Observation~4]{hhhh02} that in a graph with maximum degree at least three, a minimum power dominating set can be chosen in which each vertex has degree at least $3$. We observe that this is the same for robust power domination. 

    

    \begin{obs}\label{obs:deg3}
    Let $k\geq 0$. If $G$ is a connected graph with $\Delta(G)\geq 3$, then $G$ contains a $\gpkset$ in which every vertex has degree at least 3.
    \end{obs}
    
    A \emph{terminal path} from a vertex $v$ in $G$ is a path from $v$ to a vertex $u$ such that $\deg{u}{}=1$ and every internal vertex on the path has degree 2. A \emph{terminal cycle} from a vertex $v$ in $G$ is a cycle $v,u_1,u_2,\ldots,u_\ell,v$ in which $\deg{u_i}{G}=2$ for $i=1,\ldots, \ell$. 
    
    \begin{prop}\label{prop:twotermpathsortermcycle}
    Let $k\geq 0$ and let $G$ be a connected graph with $\Delta(G)\geq 3$. Let $S$ be a $\gpkset$ in which every vertex has degree at least 3. Any vertex $v\in S$ that has at least two terminal paths from $v$ must have $\pmus{v}=k+1$. Any vertex $v\in S$ that has at least one terminal cycle must have $\pmus{v}=k+1$.
    \end{prop}
    \begin{proof}
        Let $v$ be a vertex in $S$ and suppose that $v$ has two terminal paths or a terminal cycle. All of the vertices in the terminal paths or terminal cycle have degree 1 or 2 and so are not in $S$. Thus, there are at least two neighbors of $v$ which can only be observed via $v$. As $v$ can only observe both of these neighbors via the domination step, it must be the case that $\pmus{v} = k+1$.
    \end{proof}
    Zhao, Kang, and Chang \cite{zkc06} defined the family of graphs $\mathcal{T}$ to be those graphs obtained by taking a connected graph $H$ and for each vertex $v\in V(H)$ adding two vertices, $v'$ and $v''$; and two edges $vv'$ and $vv''$, with the edge $v'v''$ optional. The complete bipartite graph $K_{3,3}$ is the graph with vertex set $X\cup Y$ with $|X|=|Y|=3$ and edge set $E=\{xy:x\in X, y\in y\}$.
    
    \begin{thm}{\rm \cite[Theorem~3.]{zkc06}}\label{thm:powdomnover3}
    If $G$ is a connected graph on $n\geq 3$ vertices then $\gp{G}\leq\frac{n}{3}$  with equality if and only if $G\in \mathcal{T}\cup \{K_{3,3}\}$.
    \end{thm}
    
    This gives an upper bound for $\gpk{G}$ in terms of the size of the vertex set and equality conditions, as demonstrated in the next corollary.
    
    \begin{cor}
    Let $G$ be a connected graph with $n\geq 3$ vertices. Then $\gpk{G}\leq (k+1)\frac{n}{3}$ for $k\geq 0$. When $k=0$, this is an equality if and only if $G\in \mathcal{T}\cup\{K_{3,3}\}$. When $k\geq 1$, this is an equality if and only if $G\in\mathcal{T}$.
    \end{cor}
    
    \begin{proof}
    The upper bound is given by Proposition \ref{prop:basicbounds} and Theorem \ref{thm:powdomnover3}. From these results, we need only consider $\mathcal{T}\cup \{K_{3,3}\}$ for equality. The $k=0$ case follows directly from the power domination result. Let $k\geq 1$. 
    
    First consider $G\in\mathcal{T}$, constructed from $H$. Note that $\Delta \lp G \rp \geq 3$, so there exists a $\gpkset$, say $S$, in which every vertex has degree at least 3, so every vertex in $S$ is a vertex of $H$. For each $v\in V(H)$, $\deg{v}{G}\geq 3$ and there are either two terminal paths (if $v'v''\not\in E(H)$) or a terminal cycle (if  $v'v''\in E(H)$). By Proposition \ref{prop:twotermpathsortermcycle}, each $v\in V(H)$ must have at least $k+1$ PMUs.
    
    Finally, consider $K_{3,3}$. Note that $\gp{K_{3,3}}=2$. We will see in Theorem \ref{thm:k33} that $\gpk{K_{3,3}} = k +\left\lfloor \frac{k}{5}\right\rfloor +2 < 2(k+1)$ for $k\geq 1$. 
    \end{proof}
    
    \begin{defn}
    For any graph $G$, define $\bigpds{G}$ to be the size of the largest set $A\subseteq V$ such that for any $B\subseteq A$ with $|B|=\gp{G}$, $B$ is a power dominating set of $G$. 
\end{defn}
Observe that $\gp{G}\leq \bigpds{G}$. For example, the star graph $S_{16}$ shown in Figure \ref{fig:starmulti} has $\gp{S_{16}}= \bigpds{S_{16}}=1$. The complete bipartite graph $K_{3,3}$ has $\gp{K_{3,3}}= 2$ and $\bigpds{K_{3,3}}=6$ as any two vertices of $K_{3,3}$ form a power dominating set. 
    \begin{prop}\label{prop:gpklowQ}
    For any graph $G$ and $k\geq 0$,  if  $\bigpds{G} \geq k+\gp{G}$ then $\gpk{G}=k+\gp{G}$.
    \end{prop}

    \begin{proof}
    
    If $\bigpds{G} \geq k+\gp{G}$, then there exists a set $S$ of size at least $k+\gp{G}$ so that any $\gp{G}$ elements of $S$ form a power dominating set of $G$. Thus, any $\gp{G}+k$ elements of $S$ form a $k$-rPDS of $G$ of size $\gp{G}+k$ and so $\gpk{G}\leq \gp{G}+k$. By the lower bound in Proposition \ref{prop:basicbounds}, $\gpk{G}\geq \gp{G}+k$.
    \end{proof}

    When $\bigpds{G} > \gp{G} \geq 2$, the following upper bound sometimes improves the upper bound from Proposition \ref{prop:basicbounds}.

    \begin{thm}\label{thm:gpQbound}
    If $\bigpds{G} > \gp{G}\geq 2$, then $\gpk{G} \leq \left\lceil \frac{\bigpds{G}(k+\gp{G}-1)}{\bigpds{G}-\gp{G}+1} \right\rceil$ for $k\geq 1$.
    \end{thm}
    
    \begin{proof}
    Let $A=\lb v_1,v_2,\ldots,v_{\bigpds{G}}\rb\subseteq V$ be a maximum set such that any subset of size $\gp{G}$ is a power dominating set of $G$. For what follows, let $p=\left\lceil \frac{\bigpds{G}(k+\gp{G}-1)}{\bigpds{G}-\gp{G}+1} \right\rceil$. Construct
    $S=\lb v_1^{m_1}, v_2^{m_2},\ldots, v_{\bigpds{G}}^{m_{\bigpds{G}}} \rb$ where 
    \begin{align*}
    m_1= \left\lceil \frac{p}{\bigpds{G}} \right\rceil \text{ and }
    m_i= \min\lb \left\lceil \frac{p}{\bigpds{G}} \right\rceil, p-\sum_{j=1}^{i-1} m_j \rb \text{ for } i\geq 2.
    \end{align*}

    In order to show that $S$ is a $k$-rPDS of $G$, we will show that
    $p-k \geq (\gp{G}-1)\left\lceil \frac{p}{\bigpds{G}} \right\rceil + 1.$
    Assume this is true. Then whenever we have $p$ PMUs and $k$ fail, there are at least $(\gp{G}-1)\left\lceil \frac{p}{\bigpds{G}} \right\rceil + 1$ working PMUs. As each vertex has at most $\left\lceil \frac{p}{\bigpds{G}} \right\rceil$ PMUs, there are at least $\gp{G}$ vertices of $A$ that must have at least one PMU remaining and so form a power dominating set. 
    
    We prove the equivalent statement 
    \[p-k- (\gp{G}-1)\left\lceil \frac{p}{\bigpds{G}} \right\rceil \geq 1.\]
    
    \noindent Observe that by Proposition \ref{prop:ceilfracfix},
    \begin{align*}
    p-k-(\gp{G}-1)\left\lceil \frac{p}{\bigpds{G}} \right\rceil &= p-k-(\gp{G}-1)\left\lceil \frac{k+\gp{G}-1}{\bigpds{G}-\gp{G}+1} \right\rceil.
    \end{align*}
    Then by Corollary \ref{cor:ceilfuncmultbound} and simplifying, we see that
    \begin{align*}
    p-k-(\gp{G}&-1)\left\lceil \frac{p}{\bigpds{G}} \right\rceil \geq p-k-\lp \left\lceil \frac{(k+\gp{G}-1)(\gp{G}-1)}{\bigpds{G}-\gp{G}+1} \right\rceil +\gp{G}-2 \rp \label{eqn:fracfix}\\
                    &= p-\lp \left\lceil \frac{(k+\gp{G}-1)(\gp{G}-1)}{\bigpds{G}-\gp{G}+1} \right\rceil +\gp{G}-2 +k \rp\nonumber\\
                    &= p-\left\lceil \frac{(k+\gp{G}-1)(\gp{G}-1) + (\bigpds{G}-\gp{G}+1)(  \gp{G}-2+k)}{\bigpds{G}-\gp{G}+1} \right\rceil\nonumber \\
                    &=p-\left\lceil\frac{\bigpds{G}(k+\gp{G}-1) -\bigpds{G}+\gp{G}-1}{\bigpds{G}-\gp{G}+1} \right\rceil\nonumber\\
                    &=p-\left\lceil\frac{\bigpds{G}(k+\gp{G}-1)}{\bigpds{G}-\gp{G}+1} \right\rceil +1\nonumber\\
                    &=1.\nonumber \qedhere
    \end{align*}
    \end{proof}
    
    To see the difference between the upper bounds given in Theorem \ref{thm:gpQbound} and Proposition \ref{prop:basicbounds}, let $\bigpds{G}= \gp{G}+r$. Then Theorem \ref{thm:gpQbound} becomes
    \begin{align*} 
    \gpk{G} &\leq \left\lceil \frac{(\gp{G}+r)(k+\gp{G}-1)}{r+1} \right\rceil\\    
            &= \left\lceil \frac{\gp{G}(k+1)+\gp{G}(\gp{G}-2)+r(k+\gp{G}-1)}{r+1} \right\rceil \\
            &= \gp{G}(k+1)+\left\lceil \frac{-r\gp{G}(k+1)+ \gp{G}(\gp{G}-2)+ r(k+\gp{G}-1)}{r+1} \right\rceil \\
            &= \gp{G}(k+1)+\left\lceil \frac{-kr(\gp{G}-1)+\gp{G}(\gp{G}-2)-r}{r+1} \right\rceil. 
    \end{align*}
    This means that the bound from Theorem \ref{thm:gpQbound} improves the upper bound in Proposition \ref{prop:basicbounds} when the second term above is negative. 
    Since $r\geq 1$ and $\gp{G}\geq 2$, the second term is negative as $k$ approaches infinity. Thus for $\bigpds{G} > \gp{G}\geq 2$, we see that there exists some $k'$ such that for every $k\geq k'\geq 1$, the bound from Theorem \ref{thm:gpQbound} is an improvement over the upper bound from Proposition \ref{prop:basicbounds}. It will be useful to look specifically at Theorem \ref{thm:gpQbound} when $\gp{G}=2$.
    
    \begin{cor}\label{cor:2Qbound}
    If $\bigpds{G}>\gp{G} = 2$, then $\gpk{G} \leq \left\lceil \frac{\bigpds{G}(k+1)}{\bigpds{G}-1} \right\rceil$ for $k\geq 1$.
    \end{cor}
    
    To illustrate the use of the bounds in this section, we use these results to determine the $k$-robust power domination number for $K_{3,3}$ and $K_{3,b}$ for sufficiently large $b$ in Section \ref{sec:completebipartite}.

 \section{Complete bipartite graphs}\label{sec:completebipartite}

 The \emph{complete bipartite graph}, $K_{n,m}$ is the graph with vertex set $V=X\cup Y$ such that $|X|=n$, $|Y|=m$, and edge set $E=\{xy : x\in X, y\in Y\}$. We examine the case when $n=m=3$, then the case when $n=3$ and $m\geq 4$. Next, we find bounds for the $n,m\geq 4$ case, which combine to provide a result for the general $n=m$ case. We will need the following observation.
\begin{obs} \label{obs:bip_pow_set} Suppose the parts of $K_{n,m}$ are $X$ and $Y$, such that $|X| = n$ and $|Y|= m$. If $S$ is a power dominating set of $K_{n,m}$, then $S$ contains: at least 1 vertex from $X$ and 1 vertex from $Y$; or at least $n-1$ vertices from $X$; or at least $m-1$ vertices from $Y$.
\end{obs}

We will also need the Reverse Pigeonhole Principle. This follows from the the Pigeonhole Principle, which states that if $k$ objects are distributed among $n$ sets, then one set must have at least $\lc \frac{k}{n}\rc$ objects. 

\begin{rem}[Reverse Pigeonhole Principle]\label{rem:rev_pigeonhole}
    If $k$ objects are distributed among $n$ sets, then one set must have at most $\lf\frac{k}{n}\rf$ objects.
\end{rem}

To see why Remark \ref{rem:rev_pigeonhole} holds, observe that if $n$ sets each had at least $\lf\frac{k}{n}\rf+1$ objects, then, for $k = qn+r$ such that $q\geq 0$ and $0\leq r\leq n-1$, we would have
\[k \geq \sum_{i=1}^n \lp\lf\frac{k}{n}\rf +1 \rp = \sum_{i=1}^n \lp q+1 \rp =qn+n >k,\]
which is a contradiction. \\

We begin with $\gpk{K_{3,3}}$.
 \begin{thm}\label{thm:k33}
 Let $k\geq 0$. Let $K_{3,3}$ be the complete bipartite graph with parts $X=\{x_1,x_2,x_3\}$ and $Y=\{x_4,x_5,x_6\}$. Then
 \[\gpk{K_{3,3}}= k + \left\lfloor \frac{k}{5} \right\rfloor + 2.\]
 
 \end{thm}

 \begin{proof} We begin by observing that any two vertices of $K_{3,3}$ form a power dominating set, and so $\bigpds{K_{3,3}}=6$. First we prove the lower bound $k+ \left \lfloor \frac{k}{5} \right \rfloor +2 \leq \gpk{K_{3,3}}$ where $k=5q$. 
 Assume for contradiction that there exists a $\gpkotherset{5q}$ $S$ of size $5q+ \left \lfloor \frac{5q}{5} \right \rfloor +1 = 6q+1$. By the Pigeonhole Principle, some $x_i$ contains at least
 $\left\lceil\frac{6q+1}{6}\right\rceil = q + 1$
 of the PMUs. Observe that $|S|-5q=q+1$. Thus, we can remove $5q$ PMUs so that some vertex $x_i$ contains all remaining PMUs. This is a contradiction, as $\gp{K_{3,3}}=2$.
 Thus, $\gpkother{5q}{K_{3,3}} \geq 6q+2 = 5q + \left \lfloor \frac{5q}{5} \right\rfloor +2$, as desired. The lower bounds when $k$ is not a multiple of $5$ then follow by Corollary \ref{cor:incrj}.
 
 For the upper bound, observe that by Corollary \ref{cor:2Qbound}, 
 \begin{align*}
 \gpk{K_{3,3}} &\leq \left\lceil \frac{6(k+1)}{5} \right\rceil  \\
         &= k+1+\left\lceil \frac{k+1}{5} \right\rceil.  \\
 \end{align*}
 Then by Proposition \ref{prop:ceiltofloor}, we see that
 \begin{align*}
 \gpk{K_{3,3}} &\leq k+1 + \left\lfloor \frac{k+1-1}{5} \right\rfloor +1\\
         &= k + \left\lfloor \frac{k}{5} \right\rfloor +2. \qedhere
 \end{align*}
 \end{proof}

 Theorem \ref{thm:k33} gives an example of a graph for which Theorem \ref{thm:gpQbound} is tight and the structure of the $\gpkset$ suggested by the proof of Theorem \ref{thm:gpQbound} is shown in Figure \ref{fig:k33}. A larger family of complete bipartite graphs follows the same pattern, as shown in Theorem \ref{thm:k3bgeqk/3+3}.

 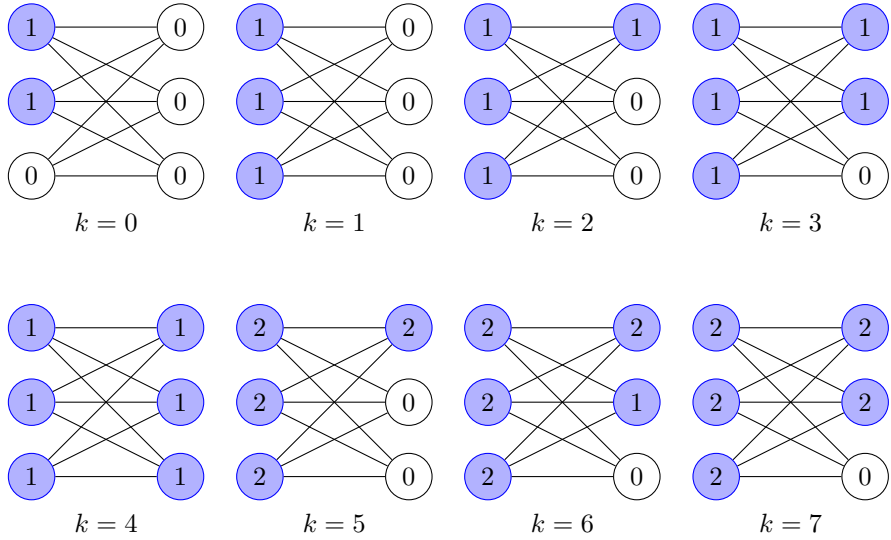
\begin{figure}[hbtp]
 \begin{center}
 \begin{tabular}{cccc}
 \begin{tikzpicture}[shorten >=9pt, shorten <=9pt, scale=0.99]
 \node[circle, draw=outline, fill=pmu] at (-1,1) {$1$};
 \node[circle, draw=outline, fill=pmu] at (-1,0) {$1$};
 \node[circle, draw=black, fill=white] at (-1,-1) {$0$};
 \node[circle, draw=black, fill=white] at (1,1) {$0$};
 \node[circle, draw=black, fill=white] at (1,0) {$0$};
 \node[circle, draw=black, fill=white] at (1,-1) {$0$};
 \draw[shorten >=9pt, shorten <=9pt] (-1,1) -- (1,1);
 \draw[shorten >=9pt, shorten <=9pt] (-1,1) -- (1,0);
 \draw[shorten >=9pt, shorten <=9pt] (-1,1) -- (1,-1);
 \draw[shorten >=9pt, shorten <=9pt] (-1,0) -- (1,1);
 \draw[shorten >=9pt, shorten <=9pt] (-1,0) -- (1,0);
 \draw[shorten >=9pt, shorten <=9pt] (-1,0) -- (1,-1);
 \draw[shorten >=9pt, shorten <=9pt] (-1,-1) -- (1,1);
 \draw[shorten >=9pt, shorten <=9pt] (-1,-1) -- (1,0);
 \draw[shorten >=9pt, shorten <=9pt] (-1,-1) -- (1,-1);
 \end{tikzpicture}
 &
 \begin{tikzpicture}[shorten >=9pt, shorten <=9pt, scale=0.99]
 \node[circle, draw=outline, fill=pmu] at (-1,1) {$1$};
 \node[circle, draw=outline, fill=pmu] at (-1,0) {$1$};
 \node[circle, draw=outline, fill=pmu] at (-1,-1) {$1$};
 \node[circle, draw=black, fill=white] at (1,1) {$0$};
 \node[circle, draw=black, fill=white] at (1,0) {$0$};
 \node[circle, draw=black, fill=white] at (1,-1) {$0$};
 \draw[shorten >=9pt, shorten <=9pt] (-1,1) -- (1,1);
 \draw[shorten >=9pt, shorten <=9pt] (-1,1) -- (1,0);
 \draw[shorten >=9pt, shorten <=9pt] (-1,1) -- (1,-1);
 \draw[shorten >=9pt, shorten <=9pt] (-1,0) -- (1,1);
 \draw[shorten >=9pt, shorten <=9pt] (-1,0) -- (1,0);
 \draw[shorten >=9pt, shorten <=9pt] (-1,0) -- (1,-1);
 \draw[shorten >=9pt, shorten <=9pt] (-1,-1) -- (1,1);
 \draw[shorten >=9pt, shorten <=9pt] (-1,-1) -- (1,0);
 \draw[shorten >=9pt, shorten <=9pt] (-1,-1) -- (1,-1);
 \end{tikzpicture}
 &
 \begin{tikzpicture}[shorten >=9pt, shorten <=9pt, scale=0.99]
 \node[circle, draw=outline, fill=pmu] at (-1,1) {$1$};
 \node[circle, draw=outline, fill=pmu] at (-1,0) {$1$};
 \node[circle, draw=outline, fill=pmu] at (-1,-1) {$1$};
 \node[circle, draw=outline, fill=pmu] at (1,1) {$1$};
 \node[circle, draw=black, fill=white] at (1,0) {$0$};
 \node[circle, draw=black, fill=white] at (1,-1) {$0$};
 \draw[shorten >=9pt, shorten <=9pt] (-1,1) -- (1,1);
 \draw[shorten >=9pt, shorten <=9pt] (-1,1) -- (1,0);
 \draw[shorten >=9pt, shorten <=9pt] (-1,1) -- (1,-1);
 \draw[shorten >=9pt, shorten <=9pt] (-1,0) -- (1,1);
 \draw[shorten >=9pt, shorten <=9pt] (-1,0) -- (1,0);
 \draw[shorten >=9pt, shorten <=9pt] (-1,0) -- (1,-1);
 \draw[shorten >=9pt, shorten <=9pt] (-1,-1) -- (1,1);
 \draw[shorten >=9pt, shorten <=9pt] (-1,-1) -- (1,0);
 \draw[shorten >=9pt, shorten <=9pt] (-1,-1) -- (1,-1);
 \end{tikzpicture}
 &
 \begin{tikzpicture}[shorten >=9pt, shorten <=9pt, scale=0.99]
 \node[circle, draw=outline, fill=pmu] at (-1,1) {$1$};
 \node[circle, draw=outline, fill=pmu] at (-1,0) {$1$};
 \node[circle, draw=outline, fill=pmu] at (-1,-1) {$1$};
 \node[circle, draw=outline, fill=pmu] at (1,1) {$1$};
 \node[circle, draw=outline, fill=pmu] at (1,0) {$1$};
 \node[circle, draw=black, fill=white] at (1,-1) {$0$};
 \draw[shorten >=9pt, shorten <=9pt] (-1,1) -- (1,1);
 \draw[shorten >=9pt, shorten <=9pt] (-1,1) -- (1,0);
 \draw[shorten >=9pt, shorten <=9pt] (-1,1) -- (1,-1);
 \draw[shorten >=9pt, shorten <=9pt] (-1,0) -- (1,1);
 \draw[shorten >=9pt, shorten <=9pt] (-1,0) -- (1,0);
 \draw[shorten >=9pt, shorten <=9pt] (-1,0) -- (1,-1);
 \draw[shorten >=9pt, shorten <=9pt] (-1,-1) -- (1,1);
 \draw[shorten >=9pt, shorten <=9pt] (-1,-1) -- (1,0);
 \draw[shorten >=9pt, shorten <=9pt] (-1,-1) -- (1,-1);
 \end{tikzpicture}
 \\
 $k=0$ & $k=1$ & $k=2$ & $k=3$ \\
 &&&\\
 &&&\\
 \begin{tikzpicture}[shorten >=9pt, shorten <=9pt, scale=0.99]
 \node[circle, draw=outline, fill=pmu] at (-1,1) {$1$};
 \node[circle, draw=outline, fill=pmu] at (-1,0) {$1$};
 \node[circle, draw=outline, fill=pmu] at (-1,-1) {$1$};
 \node[circle, draw=outline, fill=pmu] at (1,1) {$1$};
 \node[circle, draw=outline, fill=pmu] at (1,0) {$1$};
 \node[circle, draw=outline, fill=pmu] at (1,-1) {$1$};
 \draw[shorten >=9pt, shorten <=9pt] (-1,1) -- (1,1);
 \draw[shorten >=9pt, shorten <=9pt] (-1,1) -- (1,0);
 \draw[shorten >=9pt, shorten <=9pt] (-1,1) -- (1,-1);
 \draw[shorten >=9pt, shorten <=9pt] (-1,0) -- (1,1);
 \draw[shorten >=9pt, shorten <=9pt] (-1,0) -- (1,0);
 \draw[shorten >=9pt, shorten <=9pt] (-1,0) -- (1,-1);
 \draw[shorten >=9pt, shorten <=9pt] (-1,-1) -- (1,1);
 \draw[shorten >=9pt, shorten <=9pt] (-1,-1) -- (1,0);
 \draw[shorten >=9pt, shorten <=9pt] (-1,-1) -- (1,-1);
 \end{tikzpicture} 
 &
 \begin{tikzpicture}[shorten >=9pt, shorten <=9pt, scale=0.99]
 \node[circle, draw=outline, fill=pmu] at (-1,1) {$2$};
 \node[circle, draw=outline, fill=pmu] at (-1,0) {$2$};
 \node[circle, draw=outline, fill=pmu] at (-1,-1) {$2$};
 \node[circle, draw=outline, fill=pmu] at (1,1) {$2$};
 \node[circle, draw=black, fill=white] at (1,0) {$0$};
 \node[circle, draw=black, fill=white] at (1,-1) {$0$};
 \draw[shorten >=9pt, shorten <=9pt] (-1,1) -- (1,1);
 \draw[shorten >=9pt, shorten <=9pt] (-1,1) -- (1,0);
 \draw[shorten >=9pt, shorten <=9pt] (-1,1) -- (1,-1);
 \draw[shorten >=9pt, shorten <=9pt] (-1,0) -- (1,1);
 \draw[shorten >=9pt, shorten <=9pt] (-1,0) -- (1,0);
 \draw[shorten >=9pt, shorten <=9pt] (-1,0) -- (1,-1);
 \draw[shorten >=9pt, shorten <=9pt] (-1,-1) -- (1,1);
 \draw[shorten >=9pt, shorten <=9pt] (-1,-1) -- (1,0);
 \draw[shorten >=9pt, shorten <=9pt] (-1,-1) -- (1,-1);
 \end{tikzpicture} 
 &
 \begin{tikzpicture}[shorten >=9pt, shorten <=9pt, scale=0.99]
 \node[circle, draw=outline, fill=pmu] at (-1,1) {$2$};
 \node[circle, draw=outline, fill=pmu] at (-1,0) {$2$};
 \node[circle, draw=outline, fill=pmu] at (-1,-1) {$2$};
 \node[circle, draw=outline, fill=pmu] at (1,1) {$2$};
 \node[circle, draw=outline, fill=pmu] at (1,0) {$1$};
 \node[circle, draw=black, fill=white] at (1,-1) {$0$};
 \draw[shorten >=9pt, shorten <=9pt] (-1,1) -- (1,1);
 \draw[shorten >=9pt, shorten <=9pt] (-1,1) -- (1,0);
 \draw[shorten >=9pt, shorten <=9pt] (-1,1) -- (1,-1);
 \draw[shorten >=9pt, shorten <=9pt] (-1,0) -- (1,1);
 \draw[shorten >=9pt, shorten <=9pt] (-1,0) -- (1,0);
 \draw[shorten >=9pt, shorten <=9pt] (-1,0) -- (1,-1);
 \draw[shorten >=9pt, shorten <=9pt] (-1,-1) -- (1,1);
 \draw[shorten >=9pt, shorten <=9pt] (-1,-1) -- (1,0);
 \draw[shorten >=9pt, shorten <=9pt] (-1,-1) -- (1,-1);
 \end{tikzpicture} 
 &
 \begin{tikzpicture}[shorten >=9pt, shorten <=9pt, scale=0.99]
 \node[circle, draw=outline, fill=pmu] at (-1,1) {$2$};
 \node[circle, draw=outline, fill=pmu] at (-1,0) {$2$};
 \node[circle, draw=outline, fill=pmu] at (-1,-1) {$2$};
 \node[circle, draw=outline, fill=pmu] at (1,1) {$2$};
 \node[circle, draw=outline, fill=pmu] at (1,0) {$2$};
 \node[circle, draw=black, fill=white] at (1,-1) {$0$};
 \draw[shorten >=9pt, shorten <=9pt] (-1,1) -- (1,1);
 \draw[shorten >=9pt, shorten <=9pt] (-1,1) -- (1,0);
 \draw[shorten >=9pt, shorten <=9pt] (-1,1) -- (1,-1);
 \draw[shorten >=9pt, shorten <=9pt] (-1,0) -- (1,1);
 \draw[shorten >=9pt, shorten <=9pt] (-1,0) -- (1,0);
 \draw[shorten >=9pt, shorten <=9pt] (-1,0) -- (1,-1);
 \draw[shorten >=9pt, shorten <=9pt] (-1,-1) -- (1,1);
 \draw[shorten >=9pt, shorten <=9pt] (-1,-1) -- (1,0);
 \draw[shorten >=9pt, shorten <=9pt] (-1,-1) -- (1,-1);
 \end{tikzpicture} 
 \\
 $k=4$ & $k=5$ & $k=6$ & $k=7$\\
 \end{tabular}
 \end{center}
 \caption{Minimum $k$-rPDS for $K_{3,3}$ for $k=0,1,\ldots,7$.}
 \label{fig:k33}
 \end{figure}

 \begin{thm}\label{thm:k3bgeqk/3+3}
  Let $k\geq 0$. Let $K_{3,m}$ be the complete bipartite graph with parts $X= \{x_1, x_2, x_3\}$ and $Y=\{y_1, y_2, \ldots, y_m\}$. For $m\geq \left\lfloor \frac{k}{3} \right\rfloor + 3$, \[\gpk{K_{3,m}} = k + \left\lfloor \frac{k}{3} \right\rfloor +2.\] 
 \end{thm}
 
 \begin{proof}
 First we prove the lower bound when $k=3q$. Assume for eventual contradiction that there exists a $\gpkotherset{3q}$, $S$, of size $3q + \left\lfloor \frac{3q}{3} \right\rfloor +1 = 4q+1$. Let $y=\sum_{i=1}^m \pmus{y_i}$. Then 
 \[\pmus{x_1}+\pmus{x_2}+\pmus{x_3}+y = 4q+1.\]
 By the Pigeonhole Principle, we see that one of $x_1, x_2, x_3$ or $y$ must represent at least
 \[\left\lceil \frac{4m+1}{4} \right\rceil = q+1\]
 of the PMUs. Observe that $|S|-3q=q+1$. Thus, we can remove $3q$ PMUs such that either:  
 \begin{enumerate}
 
 \item All $q+1$ remaining PMUs are on a single $x_i$, which is a contradiction as this is only one vertex and $\gp{K_{3,m}}=2$; or
 
 \item All $q+1$ remaining PMUs are on the $y_i$ vertices. In order for the PMUs on the $y_i$'s to form a power dominating set of $K_{3,m}$, $m-1$ of the $y_i$'s must have a PMU. However, we also have that
 \begin{align*}
 m-1     &\geq \left\lfloor\frac{3q}{3}\right\rfloor+3-1\\
         &=q+2.
 \end{align*}
 This means that at least $q+2$ PMUs are needed but after $3q$ PMUs are removed only $q+1$ PMUs remain, a contradiction.
 \end{enumerate}
 
 Therefore, $\gpkother{3q}{K_{3,m}} > 4q+1$. Hence, $\gpkother{3q}{K_{3,m}} \geq 4q+2=3q+\left\lfloor \frac{3q}{3}\right\rfloor +2$, as desired. The lower bounds for the remaining cases then follow by Corollary \ref{cor:incrj}. 
 
 For the upper bound, the case of $k=0$ is given by the power domination number. If $m=3$ we need only consider when $k=0,1,2$; this is covered by Theorem \ref{thm:k33}. If $m\geq 4$ and $k\geq 1$, we have $\bigpds{K_{3,m}}=4$. Then by Corollary \ref{cor:2Qbound},
 \begin{align*}
 \gpk{K_{3,m}} &\leq \left\lceil \frac{4(k+1)}{3} \right\rceil\nonumber \\
         &= k+1+\left\lceil \frac{k+1}{3} \right\rceil \nonumber
 \end{align*}
 and by Proposition \ref{prop:ceiltofloor},
 \begin{align*}
 \gpk{K_{3,m}} &\leq k+1 + \left\lfloor \frac{k+1-1}{3} \right\rfloor +1\\
         &= k + \left\lfloor \frac{k}{3} \right\rfloor +2.\nonumber \qedhere
 \end{align*}
 \end{proof}
 
 

Next, we examine complete bipartite graphs with at least 4 vertices in each part, beginning with the lower bound in the following theorem.


\begin{thm}\label{thm:bipartite_lower}
Let $m \geq n\geq 4$  and $k\geq 1$. Then,  
\[\gpk{K_{n,m}} \geq 
\begin{cases}
    2(k+1)- 4\lf \frac{k}{n+2} \rf & k \equiv i \Mod{n+2},\, \, 0\leq i\leq n-4\\
    k+n-1 +(n-2)\lf \frac{k}{n+2}\rf  & \text{otherwise}
\end{cases}\]
\end{thm}


\begin{proof}
Let $K_{n,m}$ be the complete bipartite graph with $V(K_{n,m}) = X\cup Y$, leaving the sizes of $X$ and $Y$ generic. Let $k = q(n+2) +i$ for $q\geq 0$
and $0\leq i < n+2.$
Observe the following:
\begin{align*}
    i &= k - q(n+2)\\
      &= k - \lf \frac{k}{n+2}\rf(n+2)
\end{align*}

 First suppose that $0\leq i\leq n-4$. For the sake of contradiction, 
    assume that there exists a $k$-robust power dominating set $S$ such that 
    $|S| = 2k+1- 4\lf \frac{k}{n+2} \rf$. Thus,
    \begin{align*}
         |S| &= 2k+1 + (n-2 - (n+2))\lf \frac{k}{n+2} \rf\\
             &= k+1 + k - (n+2)\lf \frac{k}{n+2} \rf + (n-2)\lf \frac{k}{n+2} \rf \\
             &= k + 1 + i + (n-2)\lf \frac{k}{n+2} \rf\\
             &= q(n+2) +i + 1 + i + q(n-2)\\
             &= q(n+2)+2i +1+q(n-2)\\
             &= 2qn +2i +1.
    \end{align*} 
    Without loss of generality, assume that $\pmus{X} \leq \pmus{Y}$. By Observation 
    \ref{rem:rev_pigeonhole}, $\pmus{X} \leq \lf\frac{2qn+2i+1}{2}\rf = qn +i \leq k$. 
    Observe that if $q = 0$, we have equality.
    Let $B\subseteq S$, 
    such that $|B| = qn +i$ and $B$ contains all the PMUs from vertices of $X$ (and possibly some
    from vertices of $Y$). Then, by definition, $S\setminus B$ contains only PMUs from vertices of $Y$, and 
    $|S\setminus B| = qn +i+1$. Note that if $q= 0$, we are distributing $i+1 \leq n-3$ PMUs amongst the
    vertices of $Y$, and by Observation \ref{obs:bip_pow_set}, $S\setminus B$ is not a power dominating set.
    By Remark \ref{rem:rev_pigeonhole},
    there exists a vertex $y_1 \in Y$, such that 
    \begin{align*}
        \pmuss{S\setminus B}{y_1}&\leq \lf\frac{qn+i+1}{|Y|}\rf \\
        &\leq \lf\frac{qn+n-4+1}{n}\rf \\
        & = \lf\frac{qn}{n}+\frac{n-3}{n}\rf \\
        &= q+\lf\frac{n-3}{n}\rf \\
        &= q.
    \end{align*}
   
    Let $B'\subseteq S\setminus B$ such that $|B'|= q$ and $B'$ contains all the PMUs
    from $y_1$ (and possibly some from other vertices of $Y$). Then, by definition, $S\setminus (B\cup B')$ 
    contains PMUs only from vertices of $Y\setminus\{y_1\}$ and $|S\setminus (B\cup B')|=q(n-1)+i+1$.
    Thus, by Remark \ref{rem:rev_pigeonhole} there exists a vertex, $y_2\in Y$, such that
    \begin{align*}
        \pmuss{S\setminus (B\cup B')}{y_2} &\leq \lf\frac{q(n-1)+i+1}{|Y|-1}\rf\\
        &\leq \lf\frac{q(n-1)+n-4+1}{n-1}\rf\\
        & = \lf\frac{q(n-1)}{n-1}+\frac{n-3}{n-1}\rf \\
        &= q+\lf\frac{n-3}{n}\rf \\
        &= q.
    \end{align*}
    Let $B''\subseteq S\setminus (B\cup B')$ such that 
    $|B''| = q$ and $B$ contains all the PMUs from $y_2$ (and possibly from other vertices of $Y$).
    Then, by definition $S\setminus (B\cup B' \cup B'')$ contains only PMUs from vertices of 
    $Y\setminus\{y_1,y_2\}$. Note that  $|B\cup B' \cup B''| = k$ and by Observation \ref{obs:bip_pow_set}, 
    $S\setminus (B\cup B' \cup B'')$ is not a power
    dominating set. Therefore, no such $S$ is a $k-$robust dominating set. \\

For the second case, suppose that $n-3\leq i\leq n+1$.
 By Proposition \ref{prop:basicbounds}, to show that $\gpk{K_{n,m}}\geq k+n-1 +(n-2)\lf \frac{k}{n+2}\rf$, we need
 only show it for $i = n-3$. For the sake of contradiction, suppose there exists a $k$-robust
 power dominating set, $S$, such that $|S| = k+n-2 +(n-2)\lf \frac{k}{n+2}\rf$. Thus, 
 \begin{align*}
     |S| &= q(n+2) +n-3 +n -2 +q(n-2)\\
         &= 2qn+2n-5 
 \end{align*}
 
 Without loss of generality, assume that $\pmus{X} \leq \pmus{Y}$. 
 By Remark \ref{rem:rev_pigeonhole},
 \begin{align*}
    \pmus{X} &\leq \lf \frac{2qn+2n-5 }{2}\rf \\ 
    &= qn +n -3 \\
    &\leq q(n+2)+n-3 \\ &= k. 
 \end{align*}
 Observe that if $q = 0$, we have equality.
 Let $B\subseteq S$, such that $|B| = qn +n -3$ and $B$ contains all the PMUs from $X$ (and possibly some
 from $Y$).  
 Then by definition, $S\setminus B$ contains only PMUs from vertices of $Y$ and $|S\setminus B| = qn +n -2$.
 Note that if $q= 0$, we are distributing $n-2$ PMUs amongst the
 vertices of $Y$, and by Observation \ref{obs:bip_pow_set}, $S\setminus B$ is not a power dominating set. By Remark \ref{rem:rev_pigeonhole}, there exists a vertex, $y_1 \in Y$, such that
 \begin{align*}
  \pmuss{S\setminus B}{y_1} &\leq \lf \frac{qn+n -1}{|Y|}\rf \\
    &\leq \lf \frac{qn+n -1}{n}\rf \\
     &\leq q.
 \end{align*}
  
 Let $B'\subseteq S\setminus B$, such that $|B'| = q$ and $B'$ contains all the vertices from $y_1$ (and possibly from other vertices of $Y$). Then, by definition, $S\setminus(B\cup B')$ contains only PMUs from vertices of $Y\setminus\{y_1\}$ and $|S\setminus(B\cup B')| = q(n-1)+n - 2$. Thus, by Remark \ref{rem:rev_pigeonhole}, there exists a vertex, $y_2\in Y$, such that,
 \begin{align*}
    \pmuss{S\setminus(B\cup B')}{y_2}&\leq \lf \frac{q(n-1)+n-2}{|Y|-1}\rf\\
    &\leq \lf\frac{q(n-1)+n-2}{n-1}\rf \\
    &\leq q.
 \end{align*}

 Let $B''\subseteq S\setminus (B\cup B')$ such that $|B''|= q$ and $B$ contains all  of the PMUs from $y_2$ (and possibly from other vertices of $Y$). Then by definition,  $S\setminus (B\cup B' \cup B'')$ has only PMUs from vertices of  $Y\setminus\{y_1,y_2\}$. Note that $|B\cup B' \cup B''|= k$ and by Observation \ref{obs:bip_pow_set},  $S\setminus (B\cup B' \cup B'')$ is not a power dominating set. 
 Therefore, no such $S$ is a $k$-robust power dominating set, a contradiction for the second case.
\end{proof}

Note that the bound found in Theorem \ref{thm:bipartite_lower} is not always tight. For example, we observe that $\gpkother{4}{K_{4,6}} > 7$. To see this, suppose for contradiction that there exists a $k$-robust power dominating set $S$ such that $|S| = 7$. Let the parts of $K_{4,6}$ be $X$ and $Y$ but leave the sizes of $X$ and $Y$ generic. Then, one side, say $X$, has at most $3$ PMUs. If $Y$ has $6$ vertices, then removing all PMUs from $X$ leaves $4$ PMUs on $Y$, and so $S$ is not a $4$-rPDS. Thus, $Y$ has $4$ vertices and $X$ has 6 vertices. We then consider if we remove $4$ PMUs from $Y$. If all of the $3$ remaining PMUs are on $X$, then $S$ is not a $4$-rPDS. Thus, the PMUs must be some remaining PMUs on $Y$ and some on $X$, and so $X$ has at most $2$ PMUs. Thus, $Y$ has either $5$, $6$, or $7$ PMUs. In any case, we can remove 2 PMUs so that there are only $5$ PMUs on $Y$. However, we can still remove 2 PMUs, which leaves us with at most 2 vertices of $Y$ that contain PMUs and no vertices of $X$ containing PMUs, which is not a power dominating set of $K_{4,6}$. Therefore, we see that $\gpkother{4}{K_{4,6}} > 7$.

Next, we provide an upper bound for complete bipartite graphs.

\begin{thm}\label{thm:bipartite_upper}
Let $m\geq n\geq 4$, and $k\geq 1$ 
\[\gpk{K_{n,m}} \leq 
\begin{cases}
    2(k+1)+(m-n-4)\lf \frac{k}{n+2}\rf & k \equiv i \Mod{n+2},\, \, 0\leq i\leq n-4\\
    k+m-1 +(m-2)\lf\frac{k}{n+2}\rf  & \text{otherwise}
\end{cases}\]
\end{thm}


\begin{proof}
Suppose $V(K_{n,m}) = X\cup Y$ are the parts of $K_{n,m}$, such that $X = \{x_1,\ldots, x_n\}$ and $Y = \{y_1,\ldots, y_m\}$.
Let $k= q(n+2)+i$ for $q\geq 0$ and $0\leq i\leq n+1$. Observe that $i = k-q(n+2) = k - (n+2)\lf\frac{k}{n+2} \rf $.

For the first case, suppose that $0\leq i\leq n-4$.
To show that $\gpk{K_{n,m}} \leq 2(k+1) + (m-n-4)\lf\frac{k}{n+2}\rf$, 
it suffices to construct a $k$-robust power dominating
set of this size. First, we observe that 
\begin{align*}
     2(k+1) + (m-n-4)\lf\frac{k}{n+2}\rf &= k + 2 + k - (n+2)\lf\frac{k}{n+2}\rf +(m-2)\lf\frac{k}{n+2}\rf \\
         &= k+2 + i+ (m-2)\lf\frac{k}{n+2}\rf \\
         &= q(n+2) +i + 2 + i + q(m-2)\\
         &= q(n+2)+2i +2+q(m-2)\\
         &= q(n+m) +2(i+1)
\end{align*} 

     Let 
    \[S = \{x^{q+1}_1, \ldots, x^{q+1}_{i+1},x^{q}_{i+2},\ldots,x^{q}_{n}, 
    y^{q+1}_1, \ldots, y^{q+1}_{i+1},y^{q}_{i+2},\ldots,y^{q}_{m}\}\]
    Observe that $|S| = q(n+m) +2(i+1)$. We will now show that $S$ is a $k$-robust power dominating set. 
    Let $B\subseteq S$ such that $|B| = k = q(n+2)+i$. We have two cases:
     \begin{enumerate}
         \item $S\setminus B$ contains vertices from both $X$ and $Y$. By Observation \ref{obs:bip_pow_set}, 
         $S\setminus B$ is a power dominating set. 
         \item $S\setminus B$ contains vertices only from $X$ (or only from $Y$).
         For generality, call this side $Z$ with size $z$.  
         Observe that $|S\setminus B| = q(n+m) +2(i+1) - q(n+2) -i = q(m-2)+i+2$. 
         By Observation \ref{obs:bip_pow_set}, for $S\setminus B$ to be a robust power dominating set, it must have PMUs on at least $z-1$ vertices. Assume for contradiction that at most
         $z-2$ vertices of $Z$ have  PMUs, then,
         \begin{align*}
            |S\setminus B|&\leq (q+1)(i+1) +q(z-3-i)\\ 
            &\leq (q+1)(i+1) +q(m-3-i) \\
            &= qi+i+q+1+qm-3q-qi\\
            &=q(m-2)+i +1 \\
            &< |S\setminus B|,
         \end{align*} a contradiction.
     \end{enumerate}
     Thus $S\setminus B$ is a power dominating set, and since $B$ was arbitrary, 
     $S$ is a $k$-robust power dominating set for $K_{n,m}$.

For the second case, suppose that $n-3\leq i\leq n+1$.
 We now show that $\gpk{K_{n,m}}\leq k+m-1 +(m-2)\lf \frac{k}{n+2}\rf$ by constructing a 
 $k$-robust power dominating set
 of size $k+m-1 +(m-2)\lf \frac{k}{n+2}\rf$. By Proposition \ref{prop:basicbounds}, we need only provide the construction for
  $i = n+1$. Then,
  \begin{align*}
     k+m-1 +(m-2)\lf \frac{k}{n+2}\rf &= q(n+2)+n+1 +m-1 + q(m-2)\\
     &= q(n+m) +(n+m)\\
     &= (q+1)(n+m)
  \end{align*}
  Let $S = \{x^{q+1}_1, \ldots, x^{q+1}_n, y^{q+1}_1, \ldots, y^{q+1}_m\}$.
  Observe that $|S| = (q+1)(n+m)$. We will now show that $S$ is a $k$-robust power dominating set. 
  Let $B\subseteq S$ such that $|B| = k = q(n+2)+n+1$. We have two cases:
  \begin{enumerate}
      \item $S\setminus B$ contains vertices from both $X$ and $Y$. By Observation \ref{obs:bip_pow_set}, $S\setminus B$ is a power
      dominating set. 
      \item $S\setminus B$ contains vertices only from $X$ or only from $Y$. For generality, call this side $Z$ with size $z$.
      Observe that $|S\setminus B| =  (q+1)(n+m) -q(n+2) -n -1 = qm +m -2q -1 = q(m-2) +m -1$. 
      By Observation \ref{obs:bip_pow_set}, for $S\setminus B$ to be a robust power dominating set, it must have PMUs on at least $z-1$ vertices. Assume for contradiction that at most
      $z-2$ vertices of $Z$ have  PMUs, then, 
      \begin{align*}
        |S\setminus B| &\leq  (q+1)(z-2) \\
        &\leq (q+1)(m-2) \\
        &\leq qm +m-2q-2 \\
        &< qm+m-2q -1\\
        &= |S\setminus B|
      \end{align*}
      a contradiction.
     \end{enumerate}
     Thus $S\setminus B$ is a power dominating set, and since $B$ was arbitrary, 
     $S$ is a $k$-robust power dominating set for $K_{n,m}$.
  Thus $S$ is a $k$-robust power dominating set for $K_{n,m}$.
            \end{proof}
    
We can combine Theorem \ref{thm:bipartite_lower} and Theorem \ref{thm:bipartite_upper} to find a complete characterization for balanced complete bipartite graphs, as shown in the following corollary. Moreover, the proof of Theorem \ref{thm:bipartite_upper} yields the construction of $k$-robust power dominating sets for $K_{n,n}$        
    
    \begin{cor}{\label{bipartite_balanced}}
        Let $n\geq 4$, and $k\geq 1$ 
    \[\gpk{K_{n,n}} =
    \begin{cases}
        2(k+1)- 4\lf \frac{k}{n+2} \rf & k \equiv i \Mod{n+2},\, \, 0\leq i\leq n-3\\
        k+n-1 +(n-2)\lf \frac{k}{n+2}\rf  & \text{otherwise.}
    \end{cases}\]
    \end{cor}

\section{Trees}\label{sec:trees}
In this section, we establish the $k$-robust power domination number for trees. A \emph{tree} is an acyclic connected graph. A \emph{spider} is a tree with at most one vertex of degree 3 or more. A \emph{spider cover} of a tree $T$ is a partition of $V$, say $\{V_1,\ldots, V_\ell\}$ such that $G[V_i]$ is a spider for all $i$. The \emph{spider number} of a tree $T$, denoted by $\sp{T}$, is the minimum number of partitions in a spider cover. A \emph{rooted tree} is a tree in which one vertex is called the \textit{root}. Suppose two vertices $u$ and $v$ are in a rooted tree with root $r$. If $u$ is on the $r-v$ path, we say that $v$ is a \textit{descendant} of $u$ an $u$ is an \textit{ancestor} of $v$. If $u$ and $v$ are also neighbors, $v$ is a \emph{child} of $u$ and $v$ is the \emph{parent} of $u$. The vertex $u$ is an \textit{ancestor} of $v$ if $v$ is a descendant of $u$.

\begin{thm}{\rm \cite[Theorem 12]{hhhh02}}\label{thm:treeshhhh02}
For any tree $T$, $\gp{T} = \sp{T}$. 
\end{thm}

     {
        \begin{thm}\label{thm:trees}
            For any tree $T$, $\gpk{T} = (k+1)\sp{T}$.
        \end{thm}
        \begin{proof}
            By Proposition \ref{prop:basicbounds}, we have that $\gpk{T} \leq (k+1)\gp{T}$ for any tree $T$, and by Theorem \ref{thm:treeshhhh02} $\sp{T} = \gp{T}$. 
            Therefore, $\gpk{T} \leq (k+1)\sp{T}$.

             To prove that $\gpk{T} \geq (k+1)\sp{T}$, we proceed by contradiction. That is,
            assume that $\gpk{T}<(k+1)\sp{T}$. Since $\sp{T} = \gp{T}$, any $\gpkset$ must contain at least $\sp{T}$ distinct vertices and by the pigeonhole principle there exists at least one vertex in the set with at most $k$ PMUs. 
            
            Let $S$ be a $\gpkset$ of $T$ such that $\deg{v}{} \geq 3$ for each $v\in S$, and choose $S$ to have the smallest number of vertices $x$ with $\pmus{x}\leq k$. 
             Root $T$ at a vertex $r\in S$.
            
             Let $A = \{v\in S : 0<\pmus{v} \leq k \}$. Let $v\in A$ with the property that $d(v,r)= \max\{d(u,r): u \in A\}$. 
            Observe that for any descendant $u$ of $v$, then $\pmus{u} \geq k+1$ or $\pmus{u} = 0$. Let $w$ be the nearest ancestor of $v$ such that $\deg{w}{} \geq 3$.
            Let $S' = (S\cup\{w^{\pmuss{S}{v}}\}) \setminus \{v^{\pmuss{S}{v}}\}$, so that for each $x\in S'$ such that $x\neq w$, $\pmuss{S'}{x} = \pmuss{S}{x}$ and $\pmuss{S'}{w} = \pmuss{S}{v}+ \pmuss{S}{w}$.
            That is, $S'$ is the same as $S$ with the exception that the PMUs on $v$ have been moved to $w$. 
            
            We will show that $S'$ is a $k-$robust power dominating set. First observe that since $S\setminus \{v^{\pmuss{S}{v}}\}$ is a power dominating set,
            $S'$ is also a power dominating set. 
            Let $B'\subseteq S'$, such that $|B'|=k$. We have two cases: $\pmuss{B'}{w}\leq \pmuss{S}{v}$, or
            $\pmuss{B'}{w}> \pmuss{S}{v}$.
            
            First, assume that $\pmuss{B'}{w}\leq \pmuss{S}{v}$. Let 
            $$B = (B'\setminus \{ w^{\pmuss{B'}{w}} \}) \cup \{ v^{\pmuss{B'}{w}}\}.$$
            Note that $S\setminus B$
            removes the same PMUs as $S'\setminus B'$ with the exception that the PMUs removed from $w$ in $S'\setminus B'$, are instead removed from $v$ in $S\setminus B$.
             Since $\pmuss{B'}{w}\leq \pmuss{S}{v}$, we see that $\pmuss{S\setminus B}{w}=0$.
            
             Now, if $w \notin S'\setminus B'$, then $S\setminus B$ = $S'\setminus B'$. Thus $S'\setminus B'$ is
            a power dominating set. Next consider the case where $w \in S'\setminus B'$. 
            The ancestor $w$, which is observed by virtue of being in $S'\setminus B'$, will cause the observation of $v$ and $v$ can perform a
            zero forcing step to observe one descendant. Since $\deg{v}{} \geq 3$, all the descendants of $v$, except for possibly one, must be observed by forces from 
            descendants of $v$ in $S'\setminus{B'}$.
            To see this, recall that $S\setminus \{v^{\pmuss{S}{v}}\}$ is a power dominating set. Since $v$ is not in the power dominating set, 
            $v$ can only perform the zero forcing step.
            Since the only path from non-descendants of $v$ to descendants of $v$ is through $v$, all descendants,
            except for possibly one descendant, must be observed by descendants of $v$.
            Since $\pmuss{S\setminus B}{x} =\pmuss{S'\setminus B'}{x}$ whenever $x$ is a descendant of $v$, 
            $S'\setminus {B'}$ will force all the descendants of $v$.
            
            Thus, any vertices that rely on $\pmuss{S\setminus B}{v}>0$ to be observed in $S\setminus{B}$, will have been observed by $w$ or the descendants of $v$. 
            The remaining vertices will be observed by the same vertices as they would have been by $S\setminus{B}$. Therefore, $S'\setminus{B'}$ is a power dominating set.
    
            For the second case, assume that $\pmuss{B'}{w}> \pmuss{S}{v}$. Let 
            $$B = (B'\setminus \{w^{\pmuss{S}{v}} \})\cup \{v^{\pmuss{S}{v}} \}.$$ 
            Note that $S\setminus B$ removes the same PMUs as $S'\setminus B'$ with the exception that any PMUs removed from $w$ are first removed
            from $v$, and then the remaining are removed from $w$. Thus, $v\notin S\setminus B$, $S'\setminus B' = S\setminus B$, and therefore $S'\setminus B'$ is a power dominating set.
    
            In either case, $S'\setminus B'$ is a power dominating set. If $w\in S$, then we have found $S'$ with fewer vertices for which $\pmus{x} \leq k$ for $x \in S'$, which is a contradiction. If $w\not\in S$, 
            we may repeat the process. Since $r\in S$, we know that eventually the process terminates with the same contradiction. 
    \end{proof}
    }

    \section{Concluding remarks}\label{sec:concrem}
    
    PMU-defect-robust power domination allows us to place multiple PMUs at the same location and consider the consequences if some of these PMUs fail. There are many questions left to examine in future work.
    
    Is there an improvement to the lower bound given in Proposition \ref{prop:basicbounds} for $\gp{G}>1$? As $K_{3,3}$ demonstrates in Theorem \ref{thm:k33}, it seems likely that there is a better lower bound based on the number of vertices and the power domination number that utilizes the pigeonhole principle to show that the lower bound must increase at certain values of $k$.
    
    We have begun the study of $k$-robust power domination for certain families of graphs but work remains to be done. We have determined the $k$-robust power domination number for trees. For complete bipartite graphs, we still have the case of $\gpk{K_{3,b}}$ for $4\leq b <\left\lfloor \frac{k}{3} \right\rfloor +3$. The question of $\gpk{K_{a,b}}$ for unbalanced complete bipartite graphs when $a,b\geq 4$ is also open and preliminary observations indicate an extensive case analysis for this problem. 
    
    
    \section*{Acknowledgments}
    
   B. Bjorkman was supported by the US Department of Defense's Science, Mathematics and Research for Transformation (SMART) Scholarship for Service Program. E. Conrad was supported by the Autonomy Technology Research (ATR) Center Summer Program.
    

\bibliography{kfault}{}
\bibliographystyle{plain}
\end{document}